\documentclass[12pt,reqno]{amsart}
\usepackage{amsmath,amssymb,amsthm}
\usepackage{mathrsfs}
\usepackage{xcolor} 
\usepackage[utf8]{inputenc}
\usepackage[T1]{fontenc}
\usepackage{appendix}
\usepackage{enumitem}
\usepackage[protrusion=true,expansion=false]{microtype}
\usepackage{tikz-cd}
\usepackage[margin=1in]{geometry}
\usepackage{mathtools}
 \usepackage[colorlinks=true, linkcolor=blue, citecolor=blue]{hyperref}    
\usepackage{orcidlink}

\newcommand{\orcid}[1]{\,{\large\orcidlink{#1}}}

\newcommand{\AffiliationsBlock}{%
  \vspace{0.8em}
  \begingroup
  \centering
  \footnotesize\itshape
  $^{a}$ Department of Mathematics, Indian Institute of Technology Madras, \\
  Chennai-600036, Tamil Nadu, India.\par
  \endgroup
  \vspace{-1.3em}
}
\makeatletter
\patchcmd{\@maketitle}{
  \ifx\@empty\@dedicatory
  \else
}{
  \AffiliationsBlock
  \vspace{1.5em}
  \ifx\@empty\@dedicatory
  \else
}{}{\message{Patch failed!}}
\makeatother
\hypersetup{colorlinks=true,linkcolor=blue,citecolor=blue}
\newtheorem{theorem}{Theorem}[section]
\newtheorem{proposition}[theorem]{Proposition}
\newtheorem{lemma}[theorem]{Lemma}
\newtheorem{corollary}[theorem]{Corollary}
\theoremstyle{definition}

\newtheorem{remark}[theorem]{Remark}

\newcommand{\HP}{\mathbb{HP}}
\newcommand{\Z}{\mathbb{Z}}
\newcommand{\Q}{\mathbb{Q}}
\newcommand{\HH}{\mathbb{H}}
\newcommand{\C}{\mathbb{C}}
\newcommand{\R}{\mathbb{R}}
\newcommand{\KO}{\widetilde{KO}}
\newcommand{\Ku}{\widetilde{K}}
\newcommand{\ch}{\operatorname{ch}}
\newcommand{\ph}{\operatorname{ph}}
\newcommand{\rk}{\operatorname{rk}}
\newcommand{\ad}{\operatorname{ad}}
\newcommand{\im}{\operatorname{im}}
\newcommand{\A}{\mathcal{A}}

\begin{document}

\title[Stable complex structures on bundles over $M^{4n}_{\ell,m}$]
{Stable complex structures on real vector bundles over connected sums of
quaternionic projective spaces}

\author[S. Mandal]{{Souvik Mandal$^{\,a,*}$}
  \orcid{0009-0004-0429-3063}
}
\thanks{\hspace{-1.1em}*
  \makebox[8.5em][l]{\textit{E-mail addresses:}}%
  \texttt{ma22d014@smail.iitm.ac.in}, \texttt{ssouvik.xyz@gmail.com}}

\begin{abstract}
For integers $n\ge 2$ and integers $\ell,m\ge 0$, not both zero, let
$M^{4n}_{\ell,m}=\ell\,\mathbb{HP}^{n}\,\#\,m\,\overline{\mathbb{HP}^{n}}$ denote the connected
sum of $\ell$ copies of the quaternionic projective space $\mathbb{HP}^{n}$ and $m$
copies of $\mathbb{HP}^{n}$ endowed with the opposite orientation. By analysing the
image of the complexification map
$c\,\colon\,\widetilde{KO}(M^{4n}_{\ell,m})\to\widetilde{K}(M^{4n}_{\ell,m})$, we characterise, in
terms of Pontryagin classes, the real vector bundles over
$M^{4n}_{\ell,m}$ admitting a stable complex structure, and deduce that
$M^{4n}_{\ell,m}$ is stably almost complex if and only if $n=2$ and $\ell-m$ is
even. Consequently no $M^{4n}_{\ell,m}$ with $n\ge 3$ admits an almost
complex structure. Sato and Suzuki asserted the non-existence of almost
complex structures on $M^{4n}_{\ell,m}$ in 1974 for $3\le n\le 10$, except
possibly when $n=3$ and $\ell=3m+1$; we establish it for all $n\ge 3$, and
in the stronger form of stable almost complex structures.
\end{abstract}
\maketitle
\vspace{-0.5em}
\begin{center}
\begin{minipage}{0.845\textwidth}
    \footnotesize
    \begin{list}{}{
        \leftmargin=5.5em
        \labelwidth=5.5em
        \labelsep=0pt \parsep=0pt \topsep=0pt \itemsep=0pt
    }
        \item[\textit{Keywords:}\hfill] Stable complex structure, Almost
complex structure, Stable almost complex structure, Quaternionic
projective space, Connected sum, Complexification, Pontryagin
character, $K$-theory.
    \end{list}

    \vspace{4pt}

    \begin{list}{}{
        \leftmargin=18.5em
        \labelwidth=18.5em
        \labelsep=0pt \parsep=0pt \topsep=0pt \itemsep=0pt
    }
        \item[2020\hspace{1mm}\textit{Mathematics Subject Classification:}\hfill]
        {Primary 55R50, 57R22;\\ Secondary 19L64, 32Q60, 57R20.}
    \end{list}
\end{minipage}
\end{center}
\vspace{1em}
\section{Introduction}\label{sec:intro}
The existence of complex and stable complex structures on real vector bundles
has been studied extensively by obstruction-theoretic methods, with
characteristic classes and secondary cohomology operations as the principal
tools; see Massey \cite{Massey}, Thomas \cite{Thomas}, Yang
\cite{Yang2012,Yang2015} and Albanese and Milivojevi\'c
\cite{albanesemassey}.

No quaternionic projective space $\HP^{n}$ with $n\ge 1$ admits an almost
complex structure. For $\HP^{1}\cong S^{4}$ this is due to Ehresmann
\cite{Ehresmann} and Hopf \cite{Hopf}, and for $n\ge 4$ it was established
by Hirzebruch \cite{Hirzebruch1953} using characteristic classes. The cases
$n=2$ and $n=3$ were settled by Milnor, whose result was announced in
Hirzebruch's address to the International Congress of Mathematicians in 1958
\cite{Hirzebruch1958}; Milnor's proof, however, has remained unpublished
(see \cite[Section~1]{GMS}). Later, Massey \cite{Massey1962} gave a uniform proof of the non-existence of almost
complex structures on $\HP^{n}$ for all $n\ge 2$,
by computing the complex $K$-theory of $\HP^{n}$ and the Chern character of
its tangent bundle. For $n=1$ the stable analogue
of this non-existence result fails, since the tangent bundle of
$\HP^{1}\cong S^{4}$ is stably trivial. For $n\ge 2$, by contrast, Gauduchon,
Moroianu and Semmelmann \cite[Theorem~1.1]{GMS} proved, by means of index
theory, that a compact quaternion-K\"ahler manifold of dimension $4n$ with
positive scalar curvature is not stably almost complex unless it is isometric
to the complex Grassmannian $\operatorname{Gr}_{2}(\C^{n+2})$; in particular,
$\HP^{n}$ is not stably almost complex. For $\ell+m\ge 2$, however, the manifold
$\ell\,\HP^{n}\,\#\,m\,\overline{\HP^{n}}$ admits no such metric, and hence
lies outside the scope of \cite[Theorem~1.1]{GMS}. Indeed, it is
$3$-connected and its fourth Betti number equals $\ell+m\ge 2$, whereas that
of $\HP^{n}$ equals $1$, so that
$\ell\,\HP^{n}\,\#\,m\,\overline{\HP^{n}}$ is not diffeomorphic to
$\HP^{n}$; on the other hand, a compact quaternion-K\"ahler manifold with
positive scalar curvature and vanishing second homotopy group is
diffeomorphic to $\HP^{n}$ \cite[Theorem~0.2]{LeBrunSalamon}. 

Almost complex structures on connected sums have attracted considerable
attention. A connected sum of closed almost complex manifolds is not
necessarily almost complex, and Albanese and Milivojevi\'c
\cite{albanese2019connected} determined precisely the dimensions in which
every such connected sum is almost complex. In contrast, a connected sum of
stably almost complex manifolds is again stably almost complex, since its
tangent bundle is stably isomorphic to the Whitney sum of the pull-backs of
the tangent bundles of the summands along the collapse maps \cite{GK}.
Almost complex structures on connected sums of almost complex manifolds with
complex projective spaces were studied by Yang \cite{yang2019almost}.

We now turn our attention to connected sums of quaternionic projective
spaces. In 1974, Sato and Suzuki \cite[Theorem~B]{SatoSuzuki} asserted that
$\ell\,\HP^{n}\,\#\,m\,\overline{\HP^{n}}$ is not almost complex for every
$n\le 10$, except possibly when $n=3$ and $\ell=3m+1$. Since
then, the existence problem for almost complex structures on these manifolds
has been regarded as settled. The assertion of Sato and Suzuki is,
however, false for $n=2$. Indeed, by \cite[Cor.~1.4]{Mandal}, the manifold
$\ell\,\HP^{2}\,\#\,m\,\overline{\HP^{2}}$ carries an almost complex structure
compatible with its orientation if and only if $\ell=2m+1$ and $m$ is odd.
This criterion can also be deduced from \cite[Theorems~1 and~2]{Yang2012}
and from \cite[Theorem~1]{Heaps}. The article~\cite{Mandal} also identifies
the gap in the proof of \cite[Theorem~B]{SatoSuzuki}. For $n\ge 3$, Corollary~\ref{cor:tangent}
below shows that none of these connected sums is even stably almost complex.
This confirms the assertion of Sato and Suzuki for $3\le n\le 10$ in a
stronger form and extends it to all $n\ge 3$; in particular, the possible
exception $n=3$, $\ell=3m+1$ does not occur.

An oriented real vector bundle of even rank is said to admit a \emph{complex
structure} if it is isomorphic, as an oriented bundle, to the underlying real
bundle of a complex vector bundle. Similarly, a real vector bundle of
arbitrary rank is said to admit a \emph{stable complex structure} if it is
stably isomorphic to the underlying real bundle of a complex vector bundle.
An oriented smooth manifold of even dimension is said to be \emph{almost
complex} if its tangent bundle admits a complex structure, and a smooth
manifold of arbitrary dimension is said to be \emph{stably almost complex} if
its tangent bundle admits a stable complex structure. For a finite connected
$CW$-complex $X$, we write $\Ku(X)$ and $\KO(X)$ for its reduced complex and
real $K$-theory groups, and $r\colon\Ku(X)\to\KO(X)$ and
$c\colon\KO(X)\to\Ku(X)$ for the realification and complexification
homomorphisms, induced by regarding a complex vector bundle as a real one and
by tensoring a real vector bundle with $\C$, respectively. For such $X$, a
real vector bundle $\xi$ over $X$ admits a stable complex structure if and
only if its stable class $\tilde\xi=\xi-\rk\xi\in\KO(X)$ lies in the image of
$r$; we shall not distinguish notationally between $\xi$ and $\tilde\xi$.

The purpose of this article is to determine which real vector bundles over
the closed oriented smooth manifolds
\begin{equation}\label{eq:M}
M^{4n}_{\ell,m}\;=\;\ell\,\HP^{n}\;\#\;m\,\overline{\HP^{n}},
\qquad n\ge 2,\quad \ell,m\ge 0, \quad \ell+m\ge 1,
\end{equation}
admit a stable complex structure, and to deduce which of these manifolds are
stably almost complex; here $\overline{\HP^{n}}$ denotes $\HP^{n}$ with the
opposite orientation. Theorem~\ref{thm:main} answers the first question in
terms of the Pontryagin classes of the bundle, and
Corollary~\ref{cor:tangent} answers the second.

The case $n=1$ requires no discussion. The sphere $\mathbb{S}^{4}=\HP^{1}$
admits an orientation-reversing diffeomorphism, and it is an identity element
for the connected sum operation on closed oriented smooth $4$-manifolds.
Consequently $\ell\,\HP^{1}\,\#\,m\,\overline{\HP^{1}}$ is diffeomorphic to
$\mathbb{S}^{4}$ for all $\ell$ and $m$, and every real vector bundle over $\mathbb{S}^{4}$ admits a stable complex
structure. In other words, the realification homomorphism
$r\colon\Ku(\mathbb{S}^{4})\to\KO(\mathbb{S}^{4})$ is surjective; see
\cite[Chapter~4, Lemma~5.11]{liegrp}. For $n\ge 2$, by contrast, $\HP^{n}$ admits no orientation-reversing
diffeomorphism. Indeed, by \eqref{eq:pHP} its Pontryagin number
$\mathfrak{p}_{1}^{\,n}(\HP^{n}):=\big\langle p_{1}(\HP^{n})^{n},
[\HP^{n}]\big\rangle=(2n-2)^{n}$ is nonzero, while all Pontryagin numbers of
a manifold admitting such a diffeomorphism vanish. Thus $\HP^{n}$ and $\overline{\HP^{n}}$ are not orientation-preservingly
diffeomorphic, and $M^{4n}_{\ell,m}$ is not determined by $k=\ell+m$ alone.

We now fix the notation in which the results are stated. The fundamental
class of a closed oriented manifold $N$ is denoted by $[N]$, and we put
$k:=\ell+m$. By \cite[Corollary~2.3]{Szczarba}, 
there is a generator $u\in H^{4}(\HP^{n};\Z)$ such
that $H^{*}(\HP^{n};\Z)=\Z[u]/(u^{n+1})$ and the total Pontryagin class of
$\HP^{n}$ is
\begin{equation}\label{eq:pHP}
p(\HP^{n})=(1+u)^{2n+2}(1+4u)^{-1}.
\end{equation}
Since $p_{1}(\HP^{n})=(2n-2)u$ is nonzero for $n\ge 2$, the generator $u$ is
uniquely determined by \eqref{eq:pHP}. We orient $\HP^{n}$ so that
$\langle u^{n},[\HP^{n}]\rangle=1$. Put $k:=\ell+m$, let
$\omega_{4n,\ell,m}\in H^{4n}(M^{4n}_{\ell,m};\Z)\cong\Z$ be the generator
with $\langle\omega_{4n,\ell,m},[M^{4n}_{\ell,m}]\rangle=1$, and set
\begin{equation}\label{eq:eps}
\varepsilon_{r}=
\begin{cases} +1, & 1\le r\le \ell,\\ -1, & \ell<r\le k.\end{cases}
\end{equation}
For $1\le r\le k$, let $q_{r}\colon M^{4n}_{\ell,m}\to\HP^{n}$ denote the
standard collapse map onto the $r$-th summand, obtained by collapsing the
complement of that summand to a point; it has degree $\varepsilon_{r}$. Put
$x_{r}:=q_{r}^{*}(u)\in H^{4}(M^{4n}_{\ell,m};\Z)$.

Write $\psi\colon\mathbb{S}^{4n-1}\to\HP^{n-1}$ for the attaching map of the
top cell of $\HP^{n}$, and set
$\psi_{r}=\varepsilon_{r}\psi\in\pi_{4n-1}(\HP^{n-1})$ for $1\le r\le k$. Then up
to homotopy equivalence,
\[
M^{4n}_{\ell,m}=\Big(\bigvee_{k}\HP^{n-1}\Big)\cup_{\varphi}\mathbb{D}^{4n},
\]
where $\varphi\colon\mathbb{S}^{4n-1}\to\bigvee\limits_{k}\HP^{n-1}$ is the composite
of the pinch map $\mathbb{S}^{4n-1}\to\bigvee\limits_{k}\mathbb{S}^{4n-1}$ with
$\psi_{1}\vee\dots\vee\psi_{k}$. Thus $M^{4n}_{\ell,m}$ fits into the cofibre
sequence
\begin{equation}\label{eq:cofM}
\mathbb{S}^{4n-1}\xrightarrow{\ \varphi\ }\bigvee_{k}\HP^{n-1}
\xrightarrow{\ \iota\ }M^{4n}_{\ell,m}\xrightarrow{\ \pi\ }\mathbb{S}^{4n},
\end{equation}
where $\pi$ denotes the degree one collapse map, and $q_{r}\circ\iota$
carries the $r$-th wedge summand by the inclusion $\HP^{n-1}\subset\HP^{n}$
and collapses the remaining ones to a point.

The cohomology ring of $M^{4n}_{\ell,m}$ is given by
\begin{equation}\label{eq:ringM}
H^{*}(M^{4n}_{\ell,m};\Z)\cong\Z[x_{1},\dots ,x_{k}]\Big/
\Big(\,x_{r}x_{s},\ \varepsilon_{r}x_{r}^{\,n}-\varepsilon_{s}x_{s}^{\,n},\
x_{t}^{\,n+1}\ :\ 1\le r\neq s\le k,\ 1\le t\le k\,\Big)
\end{equation}
In particular, the Cohomologies of $M^{4n}_{\ell,m}$ are free abelian and
concentrated in the degrees $4j$ with $0\le j\le n$; for $1\le j\le n-1$ one
has $H^{4j}(M^{4n}_{\ell,m};\Z)\cong\Z^{k}$ with basis
$\{x_{1}^{j},\dots ,x_{k}^{j}\}$, while $H^{0}$ and $H^{4n}$ are infinite
cyclic and $x_{r}^{\,n}=\varepsilon_{r}\,\omega_{4n,\ell,m}$ in
$H^{4n}(M^{4n}_{\ell,m};\Z)$ for $1\le r\le k$. For $n=2$ this
description is proved in \cite[Lemma~2.3]{Mandal}, and the same argument
applies for $n\ge 3$.

For a real vector bundle $\xi$ over $M^{4n}_{\ell,m}$ we call the integers
$P_{r,j}$ $(1\le r\le k,\ 1\le j\le n-1)$ and $Q$ determined by
\begin{equation}\label{eq:coords}
p_{j}(\xi)=\sum_{r=1}^{k}P_{r,j}\,x_{r}^{\,j}\quad (1\le j\le n-1),
\qquad p_{n}(\xi)=Q\,\omega_{4n,\ell,m},
\end{equation}
the \emph{Pontryagin coordinates} of $\xi$. 

Our main result (Theorem~\ref{thm:main}) is expressed in terms of certain polynomials, which convert the Pontryagin coordinates of a bundle into the
coordinates of its complexification (Theorem~\ref{lem:translate}). Let $t_{1},t_{2},\dots$ be
indeterminates, and let $s_{i}\in\Z[t_{1},\dots ,t_{i}]$ be the \emph{Newton
polynomials}, given recursively by
\begin{equation}\label{eq:newton}
s_{1}:=t_{1},\qquad
s_{i}:=\sum_{q=1}^{i-1}(-1)^{q-1}t_{q}\,s_{i-q}+(-1)^{i-1}\,i\,t_{i}
\quad (i\ge 2).
\end{equation}
Let $\kappa_{i,j}\in\Q$ $(i,j\ge 1)$ be given by
\begin{equation}\label{eq:gamma}
\kappa_{1,j}:=\frac{2}{(2j)!}\ \ (j\ge 1),\qquad
\kappa_{i,j}:=0\ \ (j<i),\qquad
\kappa_{i,j}:=\sum_{q=1}^{\,j-i+1}\frac{2}{(2q)!}\,\kappa_{i-1,\,j-q}
\ \ (j\ge i\ge 2),
\end{equation}
and let $\A_{j}\in\Q[t_{1},\dots ,t_{j}]$ be the polynomials given by the recursion
\begin{equation}\label{eq:Arec}
\A_{1}:=t_{1},\qquad
\A_{j}:=\frac{2}{(2j)!}\,s_{j}(t_{1},\dots ,t_{j})
-\sum_{i=1}^{j-1}\kappa_{i,j}\,\A_{i}\qquad (j\ge 2).
\end{equation}

For a real vector bundle $\xi$ over $M^{4n}_{\ell,m}$ with Pontryagin
coordinates $P_{r,j}$ and $Q$ as in \eqref{eq:coords}, we define
\begin{equation}\label{eq:defB}
\mathcal{B}(\xi):=\sum_{r=1}^{k}\varepsilon_{r}\,
\A_{n}\big(P_{r,1},\dots ,P_{r,n-1},0\big)
+(-1)^{\,n-1}\,\frac{1}{(2n-1)!}\;Q .
\end{equation}
By Theorem~\ref{lem:translate}, the numbers
$\A_{j}(P_{r,1},\dots ,P_{r,j})$ $(1\le r\le k,\ 1\le j\le n-1)$ and
$\mathcal{B}(\xi)$ are integers for every real vector bundle $\xi$ over
$M^{4n}_{\ell,m}$, and by Proposition~\ref{prop:imcM} the former are even for
odd $j$, as is the latter for odd $n$.

Our main result is the following criterion for the existence of a stable
complex structure on a real vector bundle over $M^{4n}_{\ell,m}$. Its proof,
given in Section~\ref{sec:proof}, combines the $K$-theoretic characterisation
of stable complex structures in Corollary~\ref{lem:criterion}, the
description of the image of the complexification map
$c\,\colon\,\KO(M^{4n}_{\ell,m})\to\Ku(M^{4n}_{\ell,m})$ in
Proposition~\ref{prop:imcM}, and the expression of the coordinates of
$c(\xi)$ in terms of the Pontryagin coordinates of $\xi$ in
Theorem~\ref{lem:translate}.

\begin{theorem}\label{thm:main}
For $n\ge 2$, let $\xi$ be a real vector bundle over $M^{4n}_{\ell,m}$ with
Pontryagin coordinates $P_{r,j}$ and $Q$ as in \eqref{eq:coords}. Then $\xi$
admits a stable complex structure if and only if
\begin{enumerate}
\item[\textup{(a)}] $\A_{j}\big(P_{r,1},\dots ,P_{r,j}\big)$ is an even
integer for all $1\le r\le k$ and all even $j$ with $2\le j\le n-1$, and
\item[\textup{(b)}] $\mathcal{B}(\xi)$ is an even integer when $n$ is even.
\end{enumerate}
\end{theorem}
Applying Theorem~\ref{thm:main} to the tangent bundle of $M^{4n}_{\ell,m}$
yields the following corollary, proved in Section~\ref{sec:proof}.

\begin{corollary}\label{cor:tangent}
For $n\ge 2$, the manifold $M^{4n}_{\ell,m}$ is stably almost complex if and
only if $n=2$ and $\ell\equiv m\pmod 2$. In particular, $M^{4n}_{\ell,m}$ is
not almost complex for $n\ge 3$.
\end{corollary}
\begin{remark}
For $n=2$, Corollary~\ref{cor:tangent} recovers \cite[Corollary~1.2]{Mandal},
where $M^{8}_{\ell,m}$ was shown, by a different argument, to be stably
almost complex precisely when $\ell-m$ is even. By a criterion of Thomas
\cite[Theorem~1.7]{Thomas}, $M^{8}_{\ell,m}$ is almost complex if and only if
$\ell=2m+1$ and $m$ is odd~\cite[Cor.~1.4]{Mandal}.
\end{remark}
\subsection*{Organisation of the paper}
Section~\ref{sec:crit} establishes the $K$-theoretic characterisation of
stable complex structures. Section~\ref{sec:HP} computes $\Ku(\HP^{n})$ and
$\KO(\HP^{n})$ with explicit generators and determines the image of the
complexification map. Section~\ref{sec:M} carries out the corresponding
computation for $M^{4n}_{\ell,m}$. Section~\ref{sec:univ} expresses the
complexification of a real vector bundle over $M^{4n}_{\ell,m}$ in terms of
its Pontryagin classes, by means of the polynomials $\A_{j}$. Theorem~\ref{thm:main} and
Corollary~\ref{cor:tangent} are proved in Section~\ref{sec:proof}.
\section{A \texorpdfstring{$K$}{}-theoretic characterisation of stable complex
structures}\label{sec:crit}

In this section we characterise stable complex structures on real vector
bundles over a class of finite $CW$-complexes that contains the manifolds
$M^{4n}_{\ell,m}$, in terms of the complexification morphism
(Corollary~\ref{lem:criterion}). This characterisation is used in
Section~\ref{sec:proof} to prove Theorem~\ref{thm:main}.

Let $X$ be a finite connected $CW$-complex. For $y\in\Ku(X)$ we write
$\overline{y}$ for its complex conjugate. Since
$(\xi\otimes_{\R}\C)_{\R}\cong\xi\oplus\xi$ for every real vector bundle
$\xi$ and $E_{\R}\otimes_{\R}\C\cong E\oplus\overline{E}$ for every complex
vector bundle $E$, one has
\begin{equation}\label{eq:rc}
r\big(c(\xi)\big)=2\xi\quad(\xi\in\KO(X)),\qquad
c\big(r(y)\big)=y+\overline{y}\quad(y\in\Ku(X)).
\end{equation}
In the rest of this section we assume in addition that $\Ku(X)$ and
$\KO(X)$ are torsion free and that $H^{*}(X;\Q)$ is concentrated in degrees
divisible by $4$. We first show that complex conjugation is trivial on
$\Ku(X)$.
\begin{lemma}\label{lem:conj}
Complex conjugation acts as the identity on $\Ku(X)$.
\end{lemma}
\begin{proof}
For a complex vector bundle $E$ over $X$, the $k$-th Chern classes of $E$
and of its conjugate satisfy $c_{k}(\overline{E})=(-1)^{k}c_{k}(E)$
\cite[Lemma~14.9]{MS}, so the Chern roots of $\overline{E}$ are the
negatives of those of $E$. Hence the component of the Chern character
$\ch(\overline{E})$ in $H^{2i}(X;\Q)$ is $(-1)^{i}$ times that of $\ch(E)$,
and by additivity the same holds for every $y\in\Ku(X)$. Since
$H^{2i}(X;\Q)=0$ for odd $i$, it follows that $\ch(\overline{y})=\ch(y)$. As $\Ku(X)$ is torsion free, the Chern character map is injective on $\Ku(X)$
\cite[p.~19]{AH}, and hence $\overline{y}=y$.
\end{proof}
\begin{corollary}\label{lem:criterion}
A real vector bundle $\xi$ over $X$ admits a stable complex structure if and
only if $c(\xi)\in 2\Ku(X)$.
\end{corollary}
\begin{proof}
If $\xi=r(y)$ for some $y\in\Ku(X)$, then $c(\xi)=c(r(y))=y+\overline{y}=2y$
by \eqref{eq:rc} and Lemma~\ref{lem:conj}. Conversely, if $c(\xi)=2y$ for
some $y\in\Ku(X)$, then $2\xi=r(c(\xi))=2r(y)$ by \eqref{eq:rc}, so
$2(\xi-r(y))=0$; since $\KO(X)$ is torsion free, $\xi=r(y)$. Thus
$\xi\in\im(r)$ if and only if $c(\xi)\in 2\Ku(X)$.
\end{proof}
\section{\texorpdfstring{$K$}{}-theory of quaternionic projective spaces}\label{sec:HP}
In this section we compute $\Ku(\HP^{n})$ and $\KO(\HP^{n})$ with explicit
bases, together with the ring structures of $K(\HP^{n})$ and $KO(\HP^{n})$
(Propositions~\ref{prop:KHP} and~\ref{prop:KOHP}), determine the image of the
complexification map $c\colon\KO(\HP^{n})\to\Ku(\HP^{n})$
(Corollary~\ref{cor:imcHP}), and express the stable class of the tangent
bundle of $\HP^{n}$ with respect to these bases (Corollary~\ref{cor:THP}). Throughout this section $n\ge 1$ is an arbitrary integer; this generality is
needed because Proposition~\ref{prop:KOHP} is proved by induction on $n$, and
because Section~\ref{sec:M} applies Propositions~\ref{prop:KHP}
and~\ref{prop:KOHP} to $\HP^{n-1}$ as well as to $\HP^{n}$.

Let $\gamma$ be the tautological quaternionic line bundle over $\HP^{n}$
(see \cite[p.~243]{MS}); we regard its fibres as right $\HH$-modules. Right
multiplication by $i\in\HH$ makes $\gamma$ a complex vector bundle of rank
$2$ with structure group $Sp(1)=SU(2)$; in particular $c_{1}(\gamma)=0$. Let
$\gamma_{\R}$ be the underlying real vector bundle, of rank $4$, and let
$\ad\gamma$ be the real vector bundle of rank $3$ associated with the
principal $Sp(1)$-bundle of $\gamma$ via the adjoint representation of
$Sp(1)$ on its Lie algebra $\mathfrak{sp}(1)\cong\R^{3}$. Right multiplication by $j\in\HH$ is a conjugate-linear automorphism of
$\gamma$, so $\overline{\gamma}\cong\gamma$, and hence
\begin{equation}\label{eq:cx}
\gamma_{\R}\otimes_{\R}\C\cong\gamma\oplus\overline{\gamma}\cong 2\gamma.
\end{equation}

For $n\ge 2$ and the generator $u$ of $H^{4}(\HP^{n};\Z)$ fixed in
\eqref{eq:pHP}, we have
\begin{equation}\label{eq:pbund}
p(\gamma_{\R})=(1+u)^{2},\qquad p(\ad\gamma)=1+4u
\end{equation}
by \cite[Lemma~2.2 and Corollary~2.3]{Szczarba}; indeed, $\gamma_{\R}$ and
$\ad\gamma$ are the bundles denoted by $\hat{\xi}$ and $\beta(\xi)$ in
\cite{Szczarba} (for $\ad\gamma$, see \cite[Corollary~6.2]{Szczarba}), and
the generator of $H^{4}(\HP^{n};\Z)$ denoted by $v$ there coincides with $u$,
since $u$ is uniquely determined by \eqref{eq:pHP}. For $n=1$, we denote
the restriction to $\HP^{1}\subset\HP^{2}$ of the generator
$u\in H^{4}(\HP^{2};\Z)$ fixed in \eqref{eq:pHP} again by $u$; then
\eqref{eq:pbund} follows by naturality. On the other
hand, \eqref{eq:cx} and $c_{1}(\gamma)=0$ yield that
$p_{1}(\gamma_{\R})=-c_{2}(\gamma_{\R}\otimes_{\R}\C)=-2c_{2}(\gamma)$. Since
$H^{4}(\HP^{n};\Z)$ is torsion free, comparison with \eqref{eq:pbund} yields
\begin{equation}\label{eq:c2}
c_{2}(\gamma)=-u.
\end{equation}
Since $c_{1}(\gamma)=0$ and $c_{2}(\gamma)=-u$, the Chern roots of $\gamma$
are $\pm x$ with $x^{2}=u$.

We consider the classes
\begin{equation}\label{eq:gens}
\sigma:=\gamma-2\in\Ku(\HP^{n}),\qquad
\lambda:=\gamma_{\R}-4\in\KO(\HP^{n}),\qquad
\nu:=(\ad\gamma-3)-2\lambda\in\KO(\HP^{n}),
\end{equation}
and the power series
\begin{equation}\label{eq:g}
g(t):=\sum_{i\ge 1}\frac{2\,t^{i}}{(2i)!}\in\Q[[t]].
\end{equation}
Then
\begin{equation}\label{eq:chsigma}
\ch(\sigma)=e^{x}+e^{-x}-2=\sum_{i\ge 1}\frac{2\,x^{2i}}{(2i)!}
=\sum_{i\ge 1}\frac{2\,u^{i}}{(2i)!}=g(u).
\end{equation}
For every $m\ge 1$, the classes in \eqref{eq:gens} are defined over
$\HP^{m}$, and we denote them by the same symbols in each case. No ambiguity
arises, since the tautological quaternionic line bundle over $\HP^{m}$
restricts to that over $\HP^{m-1}$; consequently the inclusion
$\iota\colon\HP^{m-1}\to\HP^{m}$ satisfies
\[
\iota^{*}(\sigma)=\sigma,\qquad \iota^{*}(\lambda)=\lambda,\qquad
\iota^{*}(\nu)=\nu.
\]
\subsection{Complex \texorpdfstring{$K$}{}-theory}
Since $H^{*}(\HP^{n};\Z)$ is torsion-free and vanishes in odd degrees, the Atiyah--Hirzebruch spectral sequence for complex $K$-theory
\[
E_{2}^{p,q}=H^{p}(\HP^{n};K^{q}(\mathrm{pt}))\Longrightarrow K^{p+q}(\HP^{n})
\]
collapses at the $E_{2}$-page; that is, the differentials $d^{p,q}_{r}\colon E_{r}^{p,q}\to E_{r}^{p+r,\,q-r+1}$ vanish for all $r\ge 2$ and all $p,q$, yielding $E_{2}^{p,q}=E_{3}^{p,q}=\dots=E_{\infty}^{p,q}$ \cite[p.~18]{AH}. By Bott periodicity~\cite{Bott}, $K^{q}(\mathrm{pt})\cong\Z$ for even $q$ and $K^{q}(\mathrm{pt})=0$ for odd $q$. Consequently, the $E_{2}$-terms are free abelian and vanish whenever $p+q$ is odd, while the nonzero $E_{2}$-terms of total degree $0$ are precisely $E_{2}^{4j,-4j}= H^{4j}(\HP^{n};\Z)\cong\Z$ for $0\le j\le n$. For every integer $s$, the $E_{\infty}$-terms of total degree $s$ are the successive quotients of the skeletal filtration of $K^{s}(\HP^{n})$, whose $p$-th term is the kernel of the restriction map to the $(p-1)$-skeleton of $\HP^{n}$ in its standard cell structure. Hence, $K^{-1}(\HP^{n})=0$, and $K(\HP^{n})$ is free abelian of rank $n+1$. The following proposition determines a basis of $\Ku(\HP^{n})$; this basis
is used in the proofs of Propositions~\ref{prop:KOHP} and~\ref{prop:KM}.
\begin{proposition}\label{prop:KHP}
The group $\Ku(\HP^{n})$ is free abelian with basis
$\{\sigma,\sigma^{2},\dots,\sigma^{n}\}$, and $\sigma^{n+1}=0$. Consequently,
$K(\HP^{n})\cong\Z[\sigma]/(\sigma^{n+1})$ as rings.
\end{proposition}
\begin{proof}
For every $d\ge 1$, Bott periodicity \cite{Bott} yields
$\Ku^{1}(\mathbb{S}^{4d})=0$ and $\Ku(\mathbb{S}^{4d})\cong\Z$, and the Chern
character maps $\Ku(\mathbb{S}^{4d})$ isomorphically onto the image of
$H^{4d}(\mathbb{S}^{4d};\Z)$ in $H^{4d}(\mathbb{S}^{4d};\Q)$
\cite[Proposition~2.2]{atiyahtodd}. We prove by induction on $n$ that
$\{\sigma,\dots,\sigma^{n}\}$ is a basis of $\Ku(\HP^{n})$ and that
$\sigma^{n+1}=0$.

For $n=1$, the Chern character maps $\Ku(\HP^{1})$ isomorphically onto the
image of $H^{4}(\HP^{1};\Z)=\Z u$ in $H^{4}(\HP^{1};\Q)$, as $\HP^{1}$ is
homeomorphic to $\mathbb{S}^{4}$. Since $\ch$ is multiplicative and $u^{2}=0$,
\eqref{eq:chsigma} gives $\ch(\sigma)=g(u)=u$ and
$\ch(\sigma^{2})=g(u)^{2}=0$. Hence $\{\sigma\}$ is a basis of $\Ku(\HP^{1})$,
and $\sigma^{2}=0$ because $\ch$ is injective on $\Ku(\HP^{1})$.

Let $n\ge 2$, and assume that the assertion holds for $\HP^{n-1}$. Let
$\iota\colon\HP^{n-1}\to\HP^{n}$ be the inclusion, and let
$\pi\colon\HP^{n}\to\mathbb{S}^{4n}$ be the degree one collapse map. Since $\Ku^{-1}(\HP^{n-1})\subseteq K^{-1}(\HP^{n-1})=0$, as observed
above, and $\Ku^{1}(\mathbb{S}^{4n})=0$, the long exact sequence in reduced
$K$-theory associated with the cofibration
$\HP^{n-1}\xrightarrow{\iota}\HP^{n}\xrightarrow{\pi}\mathbb{S}^{4n}$ yields the following
short exact sequence
\begin{equation}\label{fvguef}
0\to\Ku(\mathbb{S}^{4n})\xrightarrow{\pi^{*}}\Ku(\HP^{n})
\xrightarrow{\iota^{*}}\Ku(\HP^{n-1})\to 0.
\end{equation}
As the tautological quaternionic line bundle over $\HP^{n}$ restricts to
that over $\HP^{n-1}$, we have $\iota^{*}(\sigma^{j})=\sigma^{j}$ for all
$j\ge 1$.
Hence, by the induction hypothesis, $\iota^{*}$ maps
$\{\sigma,\dots,\sigma^{n-1}\}$ to the basis $\{\sigma,\dots,\sigma^{n-1}\}$ of
$\Ku(\HP^{n-1})$ and maps $\sigma^{n}$ and $\sigma^{n+1}$ to $0$, so that
$\sigma^{n},\sigma^{n+1}\in\ker\iota^{*}=\im\pi^{*}$. Denote a generator
of $\Ku(\mathbb{S}^{4n})\cong\Z$ by $b$. Because $\im\pi^{*}$ is generated by $\pi^{*}b$, there are integers $k$ and
$k'$ such that
\begin{equation}\label{eq:sigmapi}
\sigma^{n}=k\,\pi^{*}b\qquad\text{and}\qquad\sigma^{n+1}=k'\,\pi^{*}b.
\end{equation}
Since $\ch(b)$ is the image of a generator of $H^{4n}(\mathbb{S}^{4n};\Z)$ in
$H^{4n}(\mathbb{S}^{4n};\Q)$ and
$\pi^{*}\colon H^{4n}(\mathbb{S}^{4n};\Z)\to H^{4n}(\HP^{n};\Z)=\Z u^{n}$ is
an isomorphism, we obtain, by the naturality of the Chern character,
\[
\ch(\pi^{*}b)=\pi^{*}\ch(b)=\pm u^{n}.
\]
On the other hand, by the multiplicativity of the Chern character,
\eqref{eq:chsigma} and $u^{n+1}=0$ yield
\[
\ch(\sigma^{n})=g(u)^{n}=u^{n}\qquad\text{and}\qquad
\ch(\sigma^{n+1})=g(u)^{n+1}=0.
\]
Applying $\ch$ to \eqref{eq:sigmapi} therefore yields
\[
u^{n}=\pm k\,u^{n}\qquad\text{and}\qquad 0=\pm k'\,u^{n}.
\]
As $u^{n}\neq 0$ in $H^{4n}(\HP^{n};\Q)$, it follows that $k=\pm1$ and
$k'=0$. Thus $\sigma^{n}$ generates $\ker\iota^{*}$, and $\sigma^{n+1}=0$.
Since $\Ku(\HP^{n-1})$ is free, the short exact sequence~\eqref{fvguef}
splits; hence $\{\sigma,\dots,\sigma^{n-1}\}\cup\{\sigma^{n}\}$ is a basis of $\Ku(\HP^{n})$.

Consequently, $\sigma^{n+1}=0$ yields $K(\HP^{n})\cong\Z[\sigma]/(\sigma^{n+1})$
as rings.
\end{proof}
\subsection{Real \texorpdfstring{$K$}{}-theory}
It is shown in \cite{nao}, by means of the Atiyah-Hirzebruch spectral
sequence for real $K$-theory, that $\KO(\HP^{n})$ is free abelian of rank
$n$. However, no explicit basis is computed there. An explicit basis is
needed in the proof of Proposition~\ref{prop:imcM}, and is provided by
Proposition~\ref{prop:KOHP} below. We first record the
complexifications of the classes $\lambda$ and $\nu$.
\begin{lemma}\label{lem:cgens}
The classes $\lambda,\nu\in\KO(\HP^{n})$ satisfy $c(\lambda)=2\sigma$ and
$c(\nu)=\sigma^{2}$.
\end{lemma}

\begin{proof}
By \eqref{eq:cx}, $c(\lambda)=2\gamma-4=2\sigma$. To compute $c(\nu)$,
observe that the complex vector bundle $(\ad\gamma)\otimes_{\R}\C$ has
rank $3$ and that its odd Chern classes lie in the groups
$H^{4i+2}(\HP^{n};\Z)=0$. Since moreover
$c_{2}((\ad\gamma)\otimes_{\R}\C)=-p_{1}(\ad\gamma)=-4u$ by \eqref{eq:pbund},
we obtain
\[
c\big((\ad\gamma)\otimes_{\R}\C\big)=1-4u=(1+2x)(1-2x),
\]
where $\pm x$ are the Chern roots of $\gamma$, with $x^{2}=u$. Thus $(\ad\gamma)\otimes_{\R}\C$ has Chern roots $2x,0,-2x$, and
therefore
\[
\ch\big((\ad\gamma)\otimes_{\R}\C\big)=e^{2x}+1+e^{-2x}
=\big(e^{x}+e^{-x}\big)^{2}-1=\ch(\gamma^{2}-1).
\]
Since $K(\HP^{n})$ is torsion-free, the Chern
character map is injective on $K(\HP^{n})$ \cite[p.~19]{AH}. Therefore
$(\ad\gamma)\otimes_{\R}\C=\gamma^{2}-1$ in $K(\HP^{n})$, that is,
$c(\ad\gamma-3)=\gamma^{2}-4$, and consequently
\[
c(\nu)=c(\ad\gamma-3)-2c(\lambda)=\gamma^{2}-4-4\sigma=\gamma^{2}-4-4(\gamma-2)=(\gamma-2)^{2}
=\sigma^{2}.
\]
\end{proof}
Recall that, by Bott periodicity \cite{Bott}, the coefficient groups of real
$K$-theory are
\begin{equation}\label{eq:KOpt}
KO^{-i}(\mathrm{pt})\cong
\begin{cases}
\Z, & i\equiv 0,4\pmod 8,\\
\Z/2, & i\equiv 1,2\pmod 8,\\
0, & i\equiv 3,5,6,7\pmod 8,
\end{cases}
\end{equation}
and that, for every $j\ge 1$, the complexification map
\begin{equation}\label{eq:bott}
c\colon KO^{-4j}(\mathrm{pt})\cong\Z\longrightarrow
K^{-4j}(\mathrm{pt})\cong\Z
\quad\text{is multiplication by }
a_{j}:=\begin{cases} 2, & j \text{ odd},\\ 1, & j \text{ even}\end{cases}
\end{equation}
(see the proof of \cite[Corollary~5.2]{adamsvect}). We can now determine
$\KO(\HP^{n})$ together with an explicit basis.
\begin{proposition}\label{prop:KOHP}
The group $\KO(\HP^{n})$ is free abelian with basis
\[
\big\{\,\nu^{b}\ :\ 1\le 2b\le n\,\big\}\ \cup\
\big\{\,\lambda\nu^{b}\ :\ 1\le 2b+1\le n\,\big\}.
\]
Moreover, $KO(\HP^{n})\cong\Z[\lambda,\nu]/(\lambda^{2}-4\nu,\,
\nu^{\lfloor n/2\rfloor+1},\,\lambda\nu^{\lceil n/2\rceil})$ as rings, where
$\lfloor x\rfloor$ denotes the greatest integer not exceeding $x$, and
$\lceil x\rceil$ the least integer not less than $x$.
\end{proposition}

\begin{proof}
The assertion about the basis is proved by induction on $n$; the ring
structure is determined thereafter. For $n=1$ we have
$\HP^{1}\cong\mathbb{S}^{4}$, so $\KO(\HP^{1})\cong\Z$ by \eqref{eq:KOpt},
and the complexification map
$c\colon\KO(\HP^{1})\to\Ku(\HP^{1})$ is injective with image
$2\Ku(\HP^{1})$, by \eqref{eq:bott}. Since
$c(\lambda)=2\sigma$ by Lemma~\ref{lem:cgens} and $\sigma$ generates
$\Ku(\HP^{1})$ by Proposition~\ref{prop:KHP}, the element $c(\lambda)$
generates $\im c$; hence $\lambda$ generates $\KO(\HP^{1})$.

Let $n\ge 2$, and assume that the assertion holds for $\HP^{n-1}$. Let
$\pi\colon\HP^{n}\to\mathbb{S}^{4n}$ be the degree one collapse map. Since $4n-1\equiv 3$ or $7\pmod 8$, \eqref{eq:KOpt} gives
$\KO^{1}(\mathbb{S}^{4n})\cong KO^{-(4n-1)}(\mathrm{pt})=0$; moreover,
$\KO^{-1}(\HP^{n-1})$ is a $2$-torsion group \cite{nao}, so any 
homomorphism $\KO^{-1}(\HP^{n-1})\to\KO(\mathbb{S}^{4n})\cong\Z$ is trivial.
Hence the cofibration $\HP^{n-1}\xrightarrow{\iota}\HP^{n}
\xrightarrow{\pi}\mathbb{S}^{4n}$ yields a short exact sequence
\begin{equation}\label{eq:seqHP}
0\to\KO(\mathbb{S}^{4n})\xrightarrow{\ \pi^{*}\ }\KO(\HP^{n})
\xrightarrow{\ \iota^{*}\ }\KO(\HP^{n-1})\to 0 .
\end{equation}
By \cite{nao}, the groups $\KO(\HP^{n})$ and $\KO(\HP^{n-1})$ are free
abelian of ranks $n$ and $n-1$, so $\ker\iota^{*}$ is infinite cyclic; by
exactness it equals $\im\pi^{*}$, and is therefore generated by $\pi^{*}z$,
where $z$ generates $\KO(\mathbb{S}^{4n})\cong\Z$. Since $\KO(\HP^{n})$ is free and
$c(\eta)=0$ implies $2\eta=r(c(\eta))=0$ by \eqref{eq:rc}, the map $c$ is
injective on $\KO(\HP^{n})$.

In the proof of Proposition~\ref{prop:KHP} it was shown that
$\sigma^{n}=\pm\pi^{*}b$ for a generator $b$ of $\Ku(\mathbb{S}^{4n})$. By
\eqref{eq:bott}, $c(z)=\pm a_{n}b$, and therefore, by the naturality of $c$,
\[
c(\pi^{*}z)=\pi^{*}c(z)=\pm a_{n}\sigma^{n}.
\]
Put
\[
\theta_{n}:=
\begin{cases}
\nu^{n/2}, & n\ \text{even},\\
\lambda\nu^{(n-1)/2}, & n\ \text{odd}.
\end{cases}
\]
Since $c$ is multiplicative, Lemma~\ref{lem:cgens} gives
\[
c(\theta_{n})=
\begin{cases}
\sigma^{n}, & n\ \text{even},\\
2\sigma^{n}, & n\ \text{odd},
\end{cases}
\]
so that $c(\theta_{n})=\pm c(\pi^{*}z)$, by the values of $a_{n}$ in
\eqref{eq:bott}. As $c$ is injective, $\theta_{n}=\pm\pi^{*}z$, and hence
$\theta_{n}$ generates $\ker\iota^{*}$. By the induction hypothesis, together with $\iota^{*}(\lambda)=\lambda$ and
$\iota^{*}(\nu)=\nu$, the collection
\[
\big\{\,\iota^{*}\nu^{b}\ :\ 1\le 2b\le n-1\,\big\}\ \cup\
\big\{\,\iota^{*}\lambda\nu^{b}\ :\ 1\le 2b+1\le n-1\,\big\}
\]
is a basis of $\KO(\HP^{n-1})$. Since $\KO(\HP^{n-1})$ is free abelian, the
short exact sequence \eqref{eq:seqHP} splits, and therefore
\[
\big\{\,\nu^{b}\ :\ 1\le 2b\le n-1\,\big\}\ \cup\
\big\{\,\lambda\nu^{b}\ :\ 1\le 2b+1\le n-1\,\big\}\ \cup\ \{\theta_{n}\}
\]
is a basis of $\KO(\HP^{n})$, which is the set asserted in the proposition.


It remains to determine the ring structure of $KO(\HP^{n})$. Since $c$ is
injective on $\KO(\HP^{n})$, the identities $c(\lambda^{2})=4\sigma^{2}=c(4\nu)$ and
$c(\nu^{b})=\sigma^{2b}$, $c(\lambda\nu^{b})=2\sigma^{2b+1}$, together with
$\sigma^{j}=0$ for $j>n$ (Proposition~\ref{prop:KHP}), yield
\[
\lambda^{2}=4\nu,\qquad \nu^{\lfloor n/2\rfloor+1}=0,\qquad
\lambda\nu^{\lceil n/2\rceil}=0 .
\]
Consequently,
$KO(\HP^{n})\cong\Z[\lambda,\nu]/(\lambda^{2}-4\nu,\,
\nu^{\lfloor n/2\rfloor+1},\,\lambda\nu^{\lceil n/2\rceil})$.
\end{proof}
\begin{corollary}\label{cor:imcHP}
The image of the complexification map
$c\colon\KO(\HP^{n})\to\Ku(\HP^{n})$ is
\begin{equation}\label{eq:imcHP}
\im c=\Big\{\ \sum_{j=1}^{n}A_{j}\sigma^{j}\ :\
A_{j}\in\Z\ \text{even for all odd } j\ \Big\}.
\end{equation}
\end{corollary}
\begin{proof}
By Proposition~\ref{prop:KOHP} and Lemma~\ref{lem:cgens}, the image of $c$ is
generated by
\[
c(\nu^{b})=\sigma^{2b}\quad (2b\le n),\qquad
c(\lambda\nu^{b})=2\sigma^{2b+1}\quad (2b+1\le n),
\]
that is, by $\sigma^{j}$ for even $j$ and by $2\sigma^{j}$ for odd $j$, where
$1\le j\le n$. The subgroup generated by these elements is the right-hand
side of \eqref{eq:imcHP}.
\end{proof}
We conclude this section by expressing the stable class of the tangent bundle
of $\HP^{n}$ in the basis of Proposition~\ref{prop:KOHP}. By
\cite[Theorem~2.1]{Szczarba},
\begin{equation}\label{eq:THP}
T\HP^{n}\oplus\ad\gamma\oplus\epsilon^{1}\cong(n+1)\gamma_{\R},
\end{equation}
where $\epsilon^{1}$ denotes the trivial real line bundle.
\begin{corollary}\label{cor:THP}
The stable class of the tangent bundle of $\HP^{n}$ satisfies
\begin{equation}\label{eq:TKO}
\begin{gathered}
{[T\HP^{n}]}-4n=(n-1)\lambda-\nu\quad\text{in }\KO(\HP^{n}),\\
c\big([T\HP^{n}]-4n\big)=2(n-1)\sigma-\sigma^{2}\quad\text{in }\Ku(\HP^{n}).
\end{gathered}
\end{equation}
\end{corollary}
\begin{proof}
Passing to $\KO(\HP^{n})$ in \eqref{eq:THP} yields
$[T\HP^{n}]-4n=(n+1)\lambda-(\ad\gamma-3)$, which equals $(n-1)\lambda-\nu$
by \eqref{eq:gens}. Applying $c$ and using Lemma~\ref{lem:cgens} yields the
second identity.
\end{proof}
\section{\texorpdfstring{$K$}{}-theory of \texorpdfstring{$M^{4n}_{\ell,m}$}{}}\label{sec:M}
In this section we compute $\Ku(M^{4n}_{\ell,m})$ and
$\KO(M^{4n}_{\ell,m})$ with explicit bases, together with the ring
structures of $K(M^{4n}_{\ell,m})$ and $KO(M^{4n}_{\ell,m})$
(Propositions~\ref{prop:KM} and~\ref{prop:KOM}), and determine the image of
the complexification map $c\colon\KO(M^{4n}_{\ell,m})\to\Ku(M^{4n}_{\ell,m})$ (Proposition~\ref{prop:imcM}), on which the proof of
Theorem~\ref{thm:main} is based. Throughout, $n\ge 2$ and $k=\ell+m$, as in
Section~\ref{sec:intro}. For $1\le r\le k$, put
\[
\sigma_{r}:=q_{r}^{*}\sigma\in\Ku(M^{4n}_{\ell,m}),\qquad
\lambda_{r}:=q_{r}^{*}\lambda,\quad
\nu_{r}:=q_{r}^{*}\nu\in\KO(M^{4n}_{\ell,m}),
\]
where $\sigma$, $\lambda$ and $\nu$ are the classes defined in
\eqref{eq:gens}, and $q_{r}\colon M^{4n}_{\ell,m}\to\HP^{n}$ is the collapse
map onto the $r$-th summand introduced in Section~\ref{sec:intro}.

Let $s_{4n}\in H^{4n}(\mathbb{S}^{4n};\Z)\cong\Z$ be the generator with
$\langle s_{4n},[\mathbb{S}^{4n}]\rangle=1$. Since the canonical collapse map $\pi\,\colon M^{4n}_{\ell,\,m}\rightarrow \mathbb{S}^{4n}$ has degree one, we have
$\pi^{*}s_{4n}=\omega_{4n,\ell,m}$. By \cite[Proposition~2.2]{atiyahtodd}
there is a unique generator $b_{4n}$ of $\Ku(\mathbb{S}^{4n})\cong\Z$ with
$\ch(b_{4n})=s_{4n}$. Put
\[
\tau:=\pi^{*}b_{4n}\in\Ku(M^{4n}_{\ell,m}).
\]
\subsection{Complex \texorpdfstring{$K$}{}-theory}
We first compute the Chern characters of the classes introduced above.
\begin{lemma}\label{lem:chM}
For all $r$ and all $j\ge 1$,
\[
\ch(\sigma_{r}^{\,j})=g(x_{r})^{j}
\qquad\text{and}\qquad
\ch(\tau)=\omega_{4n,\ell,m}
\]
in $H^{*}(M^{4n}_{\ell,m};\Q)$, where $g$ is the power series defined
in~\eqref{eq:g}.
\end{lemma}
\begin{proof}
By the naturality and the multiplicativity of the Chern character,
\[
\ch(\sigma_{r}^{\,j})=q_{r}^{*}\ch(\sigma)^{j}=q_{r}^{*}g(u)^{j}=g(x_{r})^{j},
\]
by \eqref{eq:chsigma} and $q_{r}^{*}u=x_{r}$, and
$\ch(\tau)=\pi^{*}\ch(b_{4n})=\pi^{*}s_{4n}=\omega_{4n,\ell,m}$.
\end{proof}

\begin{proposition}\label{prop:KM}
The group $\Ku(M^{4n}_{\ell,m})$ is free abelian of rank $k(n-1)+1$ with basis
\[
\big\{\sigma_{r}^{\,j}\ :\ 1\le r\le k,\ 1\le j\le n-1\big\}\ \cup\
\{\tau\}.
\]
Moreover, for $1\le r,s\le k$ with $r\neq s$,
\begin{equation}\label{eq:Krel}
\sigma_{r}\sigma_{s}=0,\qquad
\sigma_{r}^{\,n}=\varepsilon_{r}\tau,\qquad
\sigma_{r}^{\,n+1}=0,
\end{equation}
and consequently
\[
K(M^{4n}_{\ell,m})\cong\Z[\sigma_{1},\dots ,\sigma_{k}]\Big/
\Big(\,\sigma_{r}\sigma_{s},\ \varepsilon_{r}\sigma_{r}^{\,n}
-\varepsilon_{s}\sigma_{s}^{\,n},\ \sigma_{t}^{\,n+1}\ :\
1\le r\neq s\le k,\ 1\le t\le k\,\Big).
\]
\end{proposition}

\begin{proof}
Since $\Ku^{-1}(\bigvee\limits_{k}\HP^{n-1})=\bigoplus\limits_{k}\Ku^{-1}(\HP^{n-1})$
vanishes, as observed in Section~\ref{sec:HP}, and
$\Ku^{1}(\mathbb{S}^{4n})=0$ by Bott periodicity, applying $\Ku(\,\cdot\,)$ to the
cofibre sequence \eqref{eq:cofM} yields the short exact sequence
\begin{equation}\label{eq:seqMu}
0\to\Ku(\mathbb{S}^{4n})\xrightarrow{\ \pi^{*}\ }\Ku(M^{4n}_{\ell,m})
\xrightarrow{\ \iota^{*}\ }\bigoplus\limits_{k}\Ku(\HP^{n-1})\to 0 .
\end{equation}
By Proposition~\ref{prop:KHP}, the group $\bigoplus\limits_{k}\Ku(\HP^{n-1})$ is
free of rank $k(n-1)$; hence the short exact sequence~\eqref{eq:seqMu} splits and
$\Ku(M^{4n}_{\ell,m})$ is therefore free abelian of rank $k(n-1)+1$.

Since $q_{r}\circ\iota$ carries the $r$-th wedge summand by the inclusion
$\HP^{n-1}\subset\HP^{n}$ and collapses the remaining ones to a point, the
class $\iota^{*}\sigma_{r}^{\,j}\in\bigoplus\limits_{k}\Ku(\HP^{n-1})$ has $r$-th
component $\sigma^{j}$ and all other components zero. Hence, by
Proposition~\ref{prop:KHP},
\[
\big\{\iota^{*}\sigma_{r}^{\,j}\ :\ 1\le r\le k,\ 1\le j\le n-1\big\}
\]
is a basis of $\bigoplus\limits_{k}\Ku(\HP^{n-1})$. By exactness of \eqref{eq:seqMu},
$\ker\iota^{*}=\im\pi^{*}$ is generated by $\tau$. Since
$\bigoplus\limits_{k}\Ku(\HP^{n-1})$ is free abelian, therefore the exact sequence \eqref{eq:seqMu} splits,
and thus
\[
\big\{\sigma_{r}^{\,j}\ :\ 1\le r\le k,\ 1\le j\le n-1\big\}\ \cup\
\{\tau\}
\]
is a basis of $\Ku(M^{4n}_{\ell,m})$.

Since $\Ku(M^{4n}_{\ell,m})$ is free, the Chern character is injective on it
\cite[p.~19]{AH}; it therefore suffices to verify \eqref{eq:Krel} after
applying $\ch$. By
Lemma~\ref{lem:chM}, the multiplicativity of $\ch$ and the relations from
\eqref{eq:ringM}, we obtain
\[
\begin{gathered}
\ch(\sigma_{r}\sigma_{s})=g(x_{r})g(x_{s})=0,\qquad
\ch(\sigma_{r}^{\,n+1})=g(x_{r})^{n+1}=0,\\
\ch(\sigma_{r}^{\,n})=g(x_{r})^{n}=x_{r}^{n}
=\varepsilon_{r}\omega_{4n,\ell,m}=\varepsilon_{r}\ch(\tau),
\end{gathered}
\]
which yields \eqref{eq:Krel}.

Consequently, the relations \eqref{eq:Krel} give the asserted presentation of
the ring $K(M^{4n}_{\ell,m})$.
\end{proof}

\begin{remark}
By the injectivity of $\ch$ and Lemma~\ref{lem:chM}, the class $\tau$ is
characterised in $\Ku(M^{4n}_{\ell,m})$ by
$\ch(\tau)=\omega_{4n,\ell,m}$.
\end{remark}
\subsection{Real \texorpdfstring{$K$}{}-theory}
The computation of $\KO(M^{4n}_{\ell,m})$ proceeds along the same lines as
in complex $K$-theory, the generator of $\ker\iota^{*}$ in the exact sequence \eqref{eq:seqM}
now being provided by the classes introduced below.

For $1\le r\le k$, we define
$\theta_{r}\in\KO(M^{4n}_{\ell,m})$ by
\[
\theta_{r}:=
\begin{cases}
\nu_{r}^{\,n/2}, & n\ \text{even},\\
\lambda_{r}\nu_{r}^{\,(n-1)/2}, & n\ \text{odd}.
\end{cases}
\]
\begin{lemma}\label{lem:KOMfree}
The group $\KO(M^{4n}_{\ell,m})$ is free abelian of rank $k(n-1)+1$.
\end{lemma}

\begin{proof}
Since $\KO^{1}(\mathbb{S}^{4n})=0$ by \eqref{eq:KOpt}, and
$\KO^{-1}(\bigvee\limits_{k}\HP^{n-1})$ is a $2$-torsion group \cite{nao} while
$\KO(\mathbb{S}^{4n})\cong\Z$ is torsion free, applying $\KO(\,\cdot\,)$ to the cofibration~\eqref{eq:cofM} yields the short exact sequence
\begin{equation}\label{eq:seqM}
0\to\KO(\mathbb{S}^{4n})\xrightarrow{\ \pi^{*}\ }\KO(M^{4n}_{\ell,m})
\xrightarrow{\ \iota^{*}\ }\bigoplus_{k}\KO(\HP^{n-1})\to 0 .
\end{equation}
Moreover, $\bigoplus\limits_{k}\KO(\HP^{n-1})$ is free of rank $k(n-1)$, as
recorded in Section~\ref{sec:HP}; hence the short exact sequence~\eqref{eq:seqM} splits and
$\KO(M^{4n}_{\ell,m})$ is free abelian of rank $k(n-1)+1$.
\end{proof}

\begin{corollary}\label{cor:theta}
The class $\varepsilon_{r}\theta_{r}$ does not depend on $r$.
\end{corollary}
\begin{proof}
By Lemma~\ref{lem:cgens} and the naturality of the complexification map $c$,
\begin{equation}\label{fjvbgfejvb}
c(\lambda_{r})=2\sigma_{r},\qquad c(\nu_{r})=\sigma_{r}^{2},
\end{equation}
whence $c(\theta_{r})=a_{n}\sigma_{r}^{\,n}$ for $1\le r\le k$, where
$a_{n}\in\{1,2\}$ is the integer defined in \eqref{eq:bott}. Since
$\sigma_{r}^{\,n}=\varepsilon_{r}\tau$ by \eqref{eq:Krel}, we obtain
\begin{equation}\label{eq:ctheta}
c(\varepsilon_{r}\theta_{r})=\varepsilon_{r}a_{n}\sigma_{r}^{\,n}=a_{n}\tau .
\end{equation}
Finally, $\KO(M^{4n}_{\ell,m})$ is free by Lemma~\ref{lem:KOMfree}, so $c$
is injective on $\KO(M^{4n}_{\ell,m})$, and the assertion follows from
\eqref{eq:ctheta}.
\end{proof}
In view of Corollary~\ref{cor:theta}, we define
$\theta:=\varepsilon_{r}\theta_{r}\in\KO(M^{4n}_{\ell,m})$.

We can now compute an explicit basis of $\KO(M^{4n}_{\ell,m})$.
\begin{proposition}\label{prop:KOM}
The collection
\[
\big\{\nu_{r}^{\,b}\ :\ 1\le r\le k,\ 1\le 2b\le n-1\big\}\ \cup\
\big\{\lambda_{r}\nu_{r}^{\,b}\ :\ 1\le r\le k,\ 1\le 2b+1\le n-1\big\}\
\cup\ \{\theta\}
\]
forms a basis of $\KO(M^{4n}_{\ell,m})$. Moreover, for $1\le r,s\le k$ with
$r\neq s$,
\begin{equation}\label{eq:KOrel}
\lambda_{r}\lambda_{s}=\lambda_{r}\nu_{s}=\nu_{r}\nu_{s}=0,\qquad
\lambda_{r}^{2}=4\nu_{r},\qquad
\nu_{r}^{\,\lfloor n/2\rfloor+1}=0,\qquad
\lambda_{r}\nu_{r}^{\,\lceil n/2\rceil}=0.
\end{equation}
Consequently the ring $KO(M^{4n}_{\ell,m})$ is isomorphic to
\[
\frac{\Z[\lambda_{1},\nu_{1},\dots ,\lambda_{k},\nu_{k}]}
{\Big(\lambda_{r}\lambda_{s},\ \lambda_{r}\nu_{s},\ \nu_{r}\nu_{s},\
\varepsilon_{r}\theta_{r}-\varepsilon_{s}\theta_{s},\
\lambda_{t}^{2}-4\nu_{t},\ \nu_{t}^{\,\lfloor n/2\rfloor+1},\
\lambda_{t}\nu_{t}^{\,\lceil n/2\rceil}\ :\ 1\le r\neq s\le k,\
1\le t\le k\Big)}.
\]
\end{proposition}
\begin{proof}
Consider the classes $\iota^{*}\nu_{r}^{\,b}$ and
$\iota^{*}\lambda_{r}\nu_{r}^{\,b}$ in $\bigoplus\limits_{k}\KO(\HP^{n-1})$. By an
analogous observation as in the proof of Proposition~\ref{prop:KM}, we get
that their $r$-th components are $\nu^{b}$ and $\lambda\nu^{b}$ respectively,
all other components being zero. Hence, by Proposition~\ref{prop:KOHP},
\[
\big\{\iota^{*}\nu_{r}^{\,b}\ :\ 1\le r\le k,\ 1\le 2b\le n-1\big\}\ \cup\
\big\{\iota^{*}\lambda_{r}\nu_{r}^{\,b}\ :\ 1\le r\le k,\
1\le 2b+1\le n-1\big\}
\]
is a basis of $\bigoplus\limits_{k}\KO(\HP^{n-1})$. Let $z$ be a generator of
$\KO(\mathbb{S}^{4n})\cong\Z$. By the naturality of $c$ and \eqref{eq:bott}, we obtain
\[
c(\pi^{*}z)=\pi^{*}c(z)=\pm a_{n}\pi^{*}b_{4n}=\pm a_{n}\tau=\pm c(\theta),
\]
the last equality by \eqref{eq:ctheta}. Since $c$ is injective on $\KO(M^{4n}_{\ell,\,m})$, we have
$\theta=\pm\pi^{*}z$, so that $\theta$ generates $\im\pi^{*}=\ker\iota^{*}$
by exactness of \eqref{eq:seqM}. Since $\bigoplus\limits_{k}\KO(\HP^{n-1})$ is free abelian, the short exact sequence \eqref{eq:seqM}
splits, and therefore
\[
\big\{\nu_{r}^{\,b}\ :\ 1\le r\le k,\ 1\le 2b\le n-1\big\}\ \cup\
\big\{\lambda_{r}\nu_{r}^{\,b}\ :\ 1\le r\le k,\ 1\le 2b+1\le n-1\big\}\
\cup\ \{\theta\}
\]
is a basis of $\KO(M^{4n}_{\ell,m})$.

Applying $q_{r}^{*}$ to the relations obtained from
Proposition~\ref{prop:KOHP} yields $\nu_{r}^{\,\lfloor n/2\rfloor+1}=0$ and
$\lambda_{r}\nu_{r}^{\,\lceil n/2\rceil}=0$. Since
$c(\lambda_{r})=2\sigma_{r}$ and $c(\nu_{r})=\sigma_{r}^{2}$, while
$\sigma_{r}\sigma_{s}=0$ for $r\neq s$ by \eqref{eq:Krel}, the classes
$\lambda_{r}\lambda_{s}$, $\lambda_{r}\nu_{s}$, $\nu_{r}\nu_{s}$ and $\lambda_{r}^{2}-4\nu_{r}$ are annihilated by $c$, and thus the remaining
relations in \eqref{eq:KOrel} follow from the injectivity of $c$.

Consequently, the relations \eqref{eq:KOrel} together with the identity
$\varepsilon_{r}\theta_{r}=\theta$ of Corollary~\ref{cor:theta} yield the
asserted presentation of the ring $KO(M^{4n}_{\ell,m})$.
\end{proof}
We can now describe the image of the complexification map on
$M^{4n}_{\ell,m}$, in terms of the basis of $\Ku(M^{4n}_{\ell,m})$ furnished
by Proposition~\ref{prop:KM}.

\begin{proposition}\label{prop:imcM}
The image of the complexification map
$c\colon\KO(M^{4n}_{\ell,m})\to\Ku(M^{4n}_{\ell,m})$ is
\[
\im c=\Big\{\ \sum_{r=1}^{k}\sum_{j=1}^{n-1}A_{r,j}\sigma_{r}^{\,j}+B\tau\ :\
A_{r,j}\ \text{even for odd } j,\ \ B\ \text{even if } n\ \text{odd}\Big\}.
\]
\end{proposition}

\begin{proof}
Let $\xi\in\KO(M^{4n}_{\ell,m})$. Since $q_{r}\circ\iota\,\colon \bigvee\limits_{k}\HP^{n-1}\rightarrow\HP^{n}$ carries the $r$-th
wedge summand by the inclusion $\HP^{n-1}\subset\HP^{n}$ and collapses the
remaining ones to a point, the composite
\[
\KO(\HP^{n})\xrightarrow{\ q_{r}^{*}\ }\KO(M^{4n}_{\ell,m})
\xrightarrow{\ \iota^{*}\ }\bigoplus_{k}\KO(\HP^{n-1})
\] sends an element $\eta$ of $\KO(\HP^{n})$ to the element of
$\bigoplus\limits_{k}\KO(\HP^{n-1})$ with $r$-th component $\eta|_{\HP^{n-1}}$ and
all other components zero. By the surjectivity of
$\iota^{*}\colon\KO(\HP^{n})\to\KO(\HP^{n-1})$ in \eqref{eq:seqHP}, there
exist $\eta_{1},\dots ,\eta_{k}\in\KO(\HP^{n})$ with
$\iota^{*}\big(\sum\limits_{r=1}^{k}q_{r}^{*}\eta_{r}\big)=\iota^{*}(\xi)$; by exactness
of \eqref{eq:seqM} then there exists $z\in\KO(\mathbb{S}^{4n})$ with
\begin{equation}\label{eq:decomp}
\xi=\sum_{r=1}^{k}q_{r}^{*}(\eta_{r})+\pi^{*}(z).
\end{equation}
Thus every element of $\KO(M^{4n}_{\ell,m})$ has the form \eqref{eq:decomp},
with $\eta_{1},\dots ,\eta_{k}$ and $z$ arbitrary.

Applying $c$ to \eqref{eq:decomp} and using its naturality, we obtain
\[
c(\xi)=\sum_{r=1}^{k}q_{r}^{*}c(\eta_{r})+\pi^{*}c(z).
\]
By Corollary~\ref{cor:imcHP},
$c(\eta_{r})=\sum\limits_{j=1}^{n}A^{(r)}_{j}\sigma^{j}$ with $A^{(r)}_{j}$ even for
odd $j$. Let $z_{0}$ be the generator of $\KO(\mathbb{S}^{4n})\cong\Z$ with
$c(z_{0})=a_{n}b_{4n}$, as in \eqref{eq:bott}, and write $z=\zeta z_{0}$ with
$\zeta\in\Z$; then $\pi^{*}c(z)=a_{n}\zeta\,\tau$. Since
$q_{r}^{*}\sigma^{j}=\sigma_{r}^{\,j}$ by definition and
$\sigma_{r}^{\,n}=\varepsilon_{r}\tau$ by \eqref{eq:Krel}, we obtain
\begin{equation}\label{eq:coeffs}
c(\xi)=\sum_{r=1}^{k}\sum_{j=1}^{n-1}A^{(r)}_{j}\sigma_{r}^{\,j}
+\Big(\sum_{r=1}^{k}\varepsilon_{r}A^{(r)}_{n}+a_{n}\zeta\Big)\tau .
\end{equation}
Thus the coordinates of $c(\xi)$ are
\[
A_{r,j}:=A^{(r)}_{j}\ \ (1\le j\le n-1),\qquad
B:=\sum_{r=1}^{k}\varepsilon_{r}A^{(r)}_{n}+a_{n}\zeta .
\]
Clearly $A_{r,j}$ is even for odd $j$ and otherwise arbitrary; and if $n$ is
odd, then $A^{(r)}_{n}$ is even for every $r$ and $a_{n}=2$, thus $B$ is
even.

Conversely, consider a tuple
\[
\big(\,(A_{r,j})_{1\le r\le k,\ 1\le j\le n-1},\ B\,\big)\in\Z^{k(n-1)+1}
\]
such that $A_{r,j}$ is even for odd $j$, and $B$ is even if $n$ is odd.
Since $a_{n}=1$ for even $n$ and $a_{n}=2$ for odd $n$, there is
$\zeta\in\Z$ with $B=a_{n}\zeta$. By Corollary~\ref{cor:imcHP} there are
$\eta_{1},\dots ,\eta_{k}\in\KO(\HP^{n})$ with
$c(\eta_{r})=\sum\limits_{j=1}^{n-1}A_{r,j}\sigma^{j}$. Putting
$\xi:=\sum\limits_{r=1}^{k}q_{r}^{*}(\eta_{r})+\pi^{*}(\zeta z_{0})$, we
obtain from \eqref{eq:coeffs} that
\[
c(\xi)=\sum_{r=1}^{k}\sum_{j=1}^{n-1}A_{r,j}\sigma_{r}^{\,j}+B\tau .
\]
This proves the asserted description of $\im c$.
\end{proof}

\subsection{The tangent bundle}
We now determine the class of the tangent bundle of $M^{4n}_{\ell,m}$ and its
complexification, both needed in Section~\ref{sec:proof}. We use the
following description of the tangent bundle of a connected sum, due to
\cite[Lemma~2.1]{GK}.


\begin{theorem}\label{thm:GK}
Let $M_{1}$ and $M_{2}$ be closed oriented smooth $d$-manifolds, and let
$p_{i}\colon M_{1}\# M_{2}\to M_{i}$ be the collapse maps. Then
\[
T(M_{1}\# M_{2})\oplus\epsilon^{d}\cong
p_{1}^{*}(TM_{1})\oplus p_{2}^{*}(TM_{2}).
\]
\end{theorem}
\begin{corollary}\label{lem:TM}
In $\KO(M^{4n}_{\ell,m})$, we have
\begin{equation}\label{ehvfjfrv}
[TM^{4n}_{\ell,m}]-4n
=\sum_{r=1}^{k}\big((n-1)\,\lambda_{r}-\nu_{r}\big),
\end{equation}
and consequently
\begin{equation}\label{eq:cTM}
c\big([TM^{4n}_{\ell,m}]-4n\big)
=\sum_{r=1}^{k}\big(2(n-1)\,\sigma_{r}-\sigma_{r}^{2}\big).
\end{equation}
\end{corollary}
\begin{proof}
Applying Theorem~\ref{thm:GK} on $M^{4n}_{\ell,m}$, we obtain
\begin{equation}\label{heog}
[TM^{4n}_{\ell,m}]-4n=\sum_{r=1}^{k}q_{r}^{*}\big([T\HP^{n}]-4n\big)
\end{equation}
in $\KO(M^{4n}_{\ell,m})$. Now~\eqref{ehvfjfrv} follows from
\eqref{eq:TKO}, and \eqref{eq:cTM} from~\eqref{fjvbgfejvb}.
\end{proof}
\section{From Pontryagin classes to \texorpdfstring{$K$}{}-theoretic coordinates}\label{sec:univ}

In this section we compute the coordinates $A_{r,j}$ and $B$ of
Proposition~\ref{prop:imcM} in terms of the Pontryagin coordinates
\eqref{eq:coords} of $\xi$ (Theorem~\ref{lem:translate}), which is the last
ingredient in the proof of Theorem~\ref{thm:main}.

Recall the Pontryagin character of a real bundle $\xi$,
\[
\ph(\xi):=\ch\big(c(\xi)\big).
\]
For a complex vector bundle $\varrho$ and $i>0$, the component of
$\ch(\varrho)$ in degree $2i$ is
\[
\ch_{i}(\varrho)=\frac{1}{i!}\,
s_{i}\big(c_{1}(\varrho),\dots ,c_{i}(\varrho)\big),
\]
with $s_{i}$ as in \eqref{eq:newton} \cite[Chapter.~10]{Hir66}, while $s_{2j}\big(c_{1}(\xi\otimes\C),\dots ,c_{2j}(\xi\otimes\C)\big)
=2\,s_{j}\big(p_{1}(\xi),\dots ,p_{j}(\xi)\big)$
(this follows from the discussion after Example~16.6 on page~193 of
\cite{MS}, together with Problem~16-C there). Hence, for $j>0$,
\begin{equation}\label{eq:ph}
\ph_{j}(\xi)=\ch_{2j}(\xi\otimes\mathbb{C})=\frac{2}{(2j)!}\,s_{j}\big(p_{1}(\xi),\dots ,p_{j}(\xi)\big)
\ \in\ H^{4j}(\,\cdot\,;\Q).
\end{equation}
We first express the Pontryagin character of a bundle over
$M^{4n}_{\ell,m}$ in terms of its Pontryagin coordinates.

\begin{lemma}\label{lem:phcoords}
Let $\xi$ be a real vector bundle over $M^{4n}_{\ell,m}$ with Pontryagin
coordinates $P_{r,j}$ and $Q$ as in \eqref{eq:coords}. Then
\[
\ph_{j}(\xi)=\frac{2}{(2j)!}\sum_{r=1}^{k}
s_{j}\big(P_{r,1},\dots ,P_{r,j}\big)\,x_{r}^{\,j}
\qquad (1\le j\le n-1),
\]
and
\[
\ph_{n}(\xi)=\frac{2}{(2n)!}\Big(\sum_{r=1}^{k}\varepsilon_{r}\,
s_{n}\big(P_{r,1},\dots ,P_{r,n-1},0\big)+(-1)^{n-1}n\,Q\Big)
\,\omega_{4n,\ell,m}.
\]
\end{lemma}

\begin{proof}
Assign the weight $i$ to $t_{i}$, so that $s_{j}$ is homogeneous of weight
$j$. Accordingly, substituting $t_{i}=p_{i}(\xi)$ and then \eqref{eq:coords},
each monomial of $s_{j}$ becomes a linear combination of products
$x_{r_{1}}\cdots x_{r_{j}}$ with $1\le r_{1},\dots ,r_{j}\le k$, and such a
product vanishes by \eqref{eq:ringM} unless $r_{1}=\dots =r_{j}$. For $1\le j\le n-1$ this yields
$s_{j}(p_{1}(\xi),\dots ,p_{j}(\xi))
=\sum\limits_{r=1}^{k}s_{j}(P_{r,1},\dots ,P_{r,j})\,x_{r}^{\,j}$, and the first formula
follows from \eqref{eq:ph}.

For $j=n$, Newton's identity (the recursion \eqref{eq:newton}) yields
\[
s_{n}(t_{1},\dots ,t_{n})=s_{n}(t_{1},\dots ,t_{n-1},0)+(-1)^{n-1}n\,t_{n}.
\]
Substituting $t_{i}=p_{i}(\xi)$ and arguing as above, the first summand on
the right equals
\[
\sum\limits_{r=1}^{k}s_{n}(P_{r,1},\dots ,P_{r,n-1},0)\,x_{r}^{\,n}
=\sum\limits_{r=1}^{k}\varepsilon_{r}\,s_{n}(P_{r,1},\dots ,P_{r,n-1},0)
\,\omega_{4n,\ell,m}
\]
and the second equals
$(-1)^{n-1}n\,Q\,\omega_{4n,\ell,m}$ by \eqref{eq:coords}. The second formula
now follows from \eqref{eq:ph}.
\end{proof}

The following lemma expresses the numbers $\kappa_{i,j}$ of
\eqref{eq:gamma} as coefficients of the powers of the series $g$ of
\eqref{eq:g}. It is through this identification that the recursion
\eqref{eq:Arec} enters the proof of Theorem~\ref{lem:translate}, alongside the
classes $g(x_{r})^{j}$ of Lemma~\ref{lem:chM}.
\begin{lemma}\label{lem:gamma}
For all $i,j\ge 1$, the number $\kappa_{i,j}$ is the coefficient of $t^{j}$
in $g(t)^{i}$.
\end{lemma}

\begin{proof}
We proceed by induction on $i$. For $i=1$, the coefficient of $t^{j}$ in
$g(t)$ is $\frac{2}{(2j)!}=\kappa_{1,j}$, by \eqref{eq:g} and
\eqref{eq:gamma}. Let $i\ge 2$. Since $g$ has no constant term, the
coefficient of $t^{j}$ in $g^{i}=g\cdot g^{i-1}$ equals
$\sum\limits_{q\ge 1}\frac{2}{(2q)!}\,\kappa_{i-1,\,j-q}$ by the induction
hypothesis. As $\kappa_{i-1,\,j-q}=0$ for $j-q<i-1$, this sum equals
$\sum\limits_{q=1}^{\,j-i+1}\frac{2}{(2q)!}\,\kappa_{i-1,\,j-q}$, which is
$\kappa_{i,j}$ by \eqref{eq:gamma}. If $j<i$, this sum is empty, so the
coefficient of $t^{j}$ in $g^{i}$ is $0$, which again equals $\kappa_{i,j}$
by \eqref{eq:gamma}.
\end{proof}
Combining the two lemmas above, we can now express the coordinates of
$c(\xi)$ in terms of the Pontryagin coordinates of $\xi$.
\begin{theorem}\label{lem:translate}
Let $\xi$ be a real vector bundle over $M^{4n}_{\ell,m}$ with Pontryagin
coordinates $P_{r,j}$ and $Q$ as in \eqref{eq:coords}, and let $A_{r,j}$ and
$B$ be the coordinates of
\[
c(\xi)=\sum_{r=1}^{k}\sum_{j=1}^{n-1}A_{r,j}\,\sigma_{r}^{\,j}+B\,\tau
\]
with respect to the basis of $\Ku(M^{4n}_{\ell,m})$ given by Proposition~\ref{prop:KM}. Then, with the polynomials $\A_{j}$ of
\eqref{eq:Arec} and the invariant $\mathcal{B}(\xi)$ of \eqref{eq:defB}, we have
\begin{equation}\label{eq:Arj}
A_{r,j}=\A_{j}\big(P_{r,1},\dots ,P_{r,j}\big)\quad
(1\le r\le k,\ 1\le j\le n-1),\qquad\text{and}\qquad B=\mathcal{B}(\xi).
\end{equation}
\end{theorem}

\begin{proof}
Applying $\ch$ to $c(\xi)$ and using Lemma~\ref{lem:chM}, we obtain
\begin{equation}\label{eq:chc}
\ph(\xi)=\ch\big(c(\xi)\big)
=\sum_{r=1}^{k}\sum_{j=1}^{n-1}A_{r,j}\,g(x_{r})^{j}
+B\,\omega_{4n,\ell,m}.
\end{equation}
We compare the two sides of \eqref{eq:chc} degree by degree.

Consider first the degrees $4j$ with $1\le j\le n-1$. By
Lemma~\ref{lem:phcoords}, the degree-$4j$ part of the left-hand side of
\eqref{eq:chc} is
\[
\frac{2}{(2j)!}\sum_{r=1}^{k}s_{j}\big(P_{r,1},\dots ,P_{r,j}\big)
\,x_{r}^{\,j},
\]
while by Lemma~\ref{lem:gamma} the degree-$4j$ part of the right-hand side is
\[
\sum_{r=1}^{k}\Big(\sum_{i=1}^{j}\kappa_{i,j}A_{r,i}\Big)x_{r}^{\,j}.
\]
Since $\{x_{1}^{\,j},\dots ,x_{k}^{\,j}\}$ is a basis of
$H^{4j}(M^{4n}_{\ell,m};\Z)$, comparing coefficients yields, for each $r$,
\[
\frac{2}{(2j)!}\,s_{j}\big(P_{r,1},\dots ,P_{r,j}\big)
=\sum_{i=1}^{j}\kappa_{i,j}\,A_{r,i}.
\]
As $\kappa_{j,j}=1$ by \eqref{eq:gamma}, these equations determine the
$A_{r,j}$ recursively, and the resulting recursion is \eqref{eq:Arec}
evaluated at $t_{i}=P_{r,i}$; hence
$A_{r,j}=\A_{j}(P_{r,1},\dots ,P_{r,j})$.

In degree $4n$, we compare the coefficients of $\omega_{4n,\ell,m}$ on the
two sides of \eqref{eq:chc}, namely
\[
\frac{2}{(2n)!}\Big(\sum_{r=1}^{k}\varepsilon_{r}\,
s_{n}\big(P_{r,1},\dots ,P_{r,n-1},0\big)+(-1)^{n-1}n\,Q\Big)
\qquad\text{and}\qquad
\sum_{r=1}^{k}\varepsilon_{r}\sum_{j=1}^{n-1}\kappa_{j,n}A_{r,j}+B ,
\]
by Lemmas~\ref{lem:phcoords} and~\ref{lem:gamma} together
with $x_{r}^{\,n}=\varepsilon_{r}\omega_{4n,\ell,m}$ (see Section~\ref{sec:intro}). Equating them and solving for $B$ yields
\[
B=\sum_{r=1}^{k}\varepsilon_{r}\Big(\frac{2}{(2n)!}\,
s_{n}\big(P_{r,1},\dots ,P_{r,n-1},0\big)
-\sum_{j=1}^{n-1}\kappa_{j,n}A_{r,j}\Big)
+(-1)^{n-1}\frac{2n}{(2n)!}\,Q .
\]
Putting $t_{n}=0$ in \eqref{eq:Arec} and using the identity
$A_{r,j}=\A_{j}(P_{r,1},\dots ,P_{r,j})$ proved above yields the equality
\[
\frac{2}{(2n)!}\,s_{n}\big(P_{r,1},\dots ,P_{r,n-1},0\big)
-\sum_{j=1}^{n-1}\kappa_{j,n}A_{r,j}
=\A_{n}\big(P_{r,1},\dots ,P_{r,n-1},0\big).
\]
We thus obtain
\[
B=\sum_{r=1}^{k}\varepsilon_{r}\,\A_{n}\big(P_{r,1},\dots ,P_{r,n-1},0\big)
+(-1)^{\,n-1}\frac{1}{(2n-1)!}\,Q =\mathcal{B}(\xi).
\]
\end{proof}
\section{Existence of stable complex structures on bundles over
\texorpdfstring{$M^{4n}_{\ell,m}$}{}}\label{sec:proof}

We now combine the computations of the preceding sections into a proof of
Theorem~\ref{thm:main}, and then deduce Corollary~\ref{cor:tangent} by
applying the theorem to the tangent bundle of $M^{4n}_{\ell,m}$.

\begin{proof}[\textbf{Proof of Theorem~\ref{thm:main}}]
Since the cohomology groups of $M^{4n}_{\ell,m}$ are torsion free and
concentrated in degrees divisible by $4$, while the groups
$\Ku(M^{4n}_{\ell,m})$ and $\KO(M^{4n}_{\ell,m})$ are free abelian by
Proposition~\ref{prop:KM} and Lemma~\ref{lem:KOMfree},
Corollary~\ref{lem:criterion} implies that $\xi$ admits a stable complex structure
if and only if $c(\xi)\in 2\Ku(M^{4n}_{\ell,m})$, that is, if and only if
all coordinates of $c(\xi)$ with respect to the basis of
Proposition~\ref{prop:KM} are even.

By Theorem~\ref{lem:translate}, these coordinates are the integers
$\A_{j}(P_{r,1},\dots ,P_{r,j})$, for $1\le r\le k$ and $1\le j\le n-1$, and
$\mathcal{B}(\xi)$. By Proposition~\ref{prop:imcM}, the integers
$\A_{j}(P_{r,1},\dots ,P_{r,j})$ with $j$ odd are even, and
$\mathcal{B}(\xi)$ is even when $n$ is odd.  Hence all the coordinates of
$c(\xi)$ are even precisely when conditions \textup{(a)} and \textup{(b)} hold, and the
theorem follows.
\end{proof}

\begin{proof}[\textbf{Proof of Corollary~\ref{cor:tangent}}]
Write $P_{r,j}$ for the Pontryagin coordinates of $TM^{4n}_{\ell,m}$.

Let $n\ge 3$. Then $\sigma_{r}^{2}$ belongs to the basis of
Proposition~\ref{prop:KM}, so \eqref{ehvfjfrv}, together with
Theorem~\ref{lem:translate}, yields
\[
\A_{2}\big(P_{r,1},P_{r,2}\big)=-1 \qquad (1\le r\le k).
\]
Thus condition \textup{(a)} of Theorem~\ref{thm:main} fails for $j=2$, and
consequently $M^{4n}_{\ell,m}$ admits no stable almost complex structure.

Let now $n=2$. The range $2\le j\le n-1$ in condition \textup{(a)} of
Theorem~\ref{thm:main} is then empty, and thus $M^{8}_{\ell,m}$ is stably
almost complex if and only if $\mathcal{B}(TM^{8}_{\ell,m})$ is an even
integer. By \eqref{eq:pHP}, $p(\HP^{2})=1+2u+7u^{2}$, so \eqref{heog} yields the Pontryagin coordinates of $TM^{8}_{\ell,m}$ as
\[
P_{r,1}=2\quad (1\le r\le k),\qquad Q=7(\ell-m).
\]
The recursion \eqref{eq:Arec} yields
$\A_{2}(t_{1},t_{2})=\tfrac{1}{12}\big(t_{1}^{2}-t_{1}-2t_{2}\big)$, thus
$\A_{2}(P_{r,1},0)=\tfrac{1}{6}$, and \eqref{eq:defB} yields
\[
\mathcal{B}\big(TM^{8}_{\ell,m}\big)
=\frac{1}{6}\sum_{r=1}^{k}\varepsilon_{r}-\frac{1}{6}\,Q=-(\ell-m),
\]
which is even precisely when $\ell\equiv m\pmod 2$. Together with the case
$n\ge 3$, this proves the corollary.
\end{proof}
\section*{Acknowledgement}

This work will be included as part of the author's doctoral thesis. The author is grateful for financial support in the form of the Prime Minister's
Research Fellowship, Government of India (PMRF/2502403).
\bibliographystyle{alpha} 
\bibliography{references}
\end{document}